\documentclass{article}
\usepackage[a4paper,top=3cm,bottom=3.5cm,left=3.5cm,right=3.5cm,marginparwidth=1.75cm]{geometry}
\usepackage[hidelinks]{hyperref}
\usepackage{amsmath, amssymb, amsthm, bbm, siunitx, comment, caption, cleveref, xfrac, subfigure}
\usepackage{url}
\usepackage{graphicx, wrapfig, tikz, array}
\usepackage{xcolor}
\usepackage{caption}
\usepackage[symbol]{footmisc}
\graphicspath{ {./figures/} }
\usetikzlibrary{patterns}

\usepackage[sort&compress,numbers]{natbib}

\newtheorem{theorem}{Theorem}
\numberwithin{theorem}{section}
\newtheorem{prop}[theorem]{Proposition}
\newtheorem{lemma}[theorem]{Lemma}
\newtheorem{defn}[theorem]{Definition}

\newtheorem{remark}[theorem]{Remark}
\newtheorem{example}[theorem]{Example}

\newcommand{\ddp}[2]{\frac{\partial#1}{\partial#2}}
\newcommand{\D}{\partial D}
\newcommand{\Si}{\mathcal{S}}
\newcommand{\A}{\mathcal{A}}
\newcommand{\K}{\mathcal{K}}
\newcommand{\spn}{\text{span}}
\newcommand{\outside}{\mathbb{R}^2\setminus \overline{D}}
\newcommand{\sinc}{\text{sinc}}

\title{A fully-coupled subwavelength resonance approach to modelling the passive cochlea}
\author{Habib Ammari\thanks{Department of Mathematics, ETH Z\"urich, R\"amistrasse 101, CH-8092 Z\"urich, Switzerland (habib.ammari@math.ethz.ch, bryn.davies@sam.math.ethz.ch).} \and Bryn Davies\footnotemark[1]}
\date{}

\begin{document}
\maketitle
\noindent\rule{\textwidth}{0.4pt}
\begin{abstract}
    The aim of this paper is to understand the behaviour of a large number of coupled subwavelength resonators. We use layer potential techniques in combination with numerical computations to study {the acoustic pressure field due to scattering by} a graded array of subwavelength resonators. Using this method, we study a graded-resonance model for the cochlea. We compute the resonant modes of the system and explore the model's ability to decompose incoming signals. We are able to offer mathematical explanations for the cochlea's so-called ``travelling wave'' behaviour and tonotopic frequency map.
\end{abstract}



\noindent\rule{\textwidth}{0.4pt}
\section{Introduction}

The development of the understanding of the cochlea has largely been a dichotomy between two classes of models \cite{bell2012resonance}. The first, proposed by Hermann von Helmholtz in the 1850s, is based on resonators tuned to different audible frequencies being distributed along the length of the cochlea \cite{helmholtz1875}. Later, Georg von B\'ek\'esy demonstrated that when the cochlea is stimulated a wave travels from the base to the apex along the basilar membrane \cite{von1960experiments}. This discovery won him a Nobel Prize in 1961 and lead to the creation of models based on each receptor cell being excited in sequence as the signal travels through the cochlea.

The cochlea is, at its simplest, a long tube filled with fluid, into which sound waves enter through the oval window. An elastic membrane, known as the basilar membrane, is suspended across the centre and upon its surface sit bundles of cylindrical cells, known as hair cells. These bundles of hair cells are the receptor cells of the ear, which produce electrical signals when deflected laterally \cite{hudspeth1983hair,hudspeth2008making}. The tips of the hair cells are attached to a membrane known as the tectorial membrane and, as a result, motion of the basilar membrane displaces the hair cell and a signal is produced \cite{reichenbach2014physics}. The cochlea's ability to filter sounds by pitch is based on the fact that the basilar membrane is graded in both size and stiffness. Thus, cochlear mechanics is, at its heart, a question of studying the motion of a graded elastic membrane in response to a sound wave.

\begin{figure}
	\begin{center}
		\begin{tikzpicture}
		\draw[fill=blue!20!white] (0,0.7) -- (0,1) -- plot [smooth] coordinates {(0,1) (7,0.9) (8,0) (7,-0.9) (0,-1)} -- (0,0.4) -- plot [smooth] coordinates {(0,0.4) (0.05,0.55) (0,0.7)};
		\draw[gray] plot [smooth] coordinates {(0,0.4) (-0.05,0.55) (0,0.7)};
		
		\begin{scope}[xshift=5,scale=0.2]
		\draw[fill=gray] (0.6,0) -- plot [smooth] coordinates {(0,0) (0.3,1) (0.6,0)};
		\end{scope}
		
		\begin{scope}[xshift=15,scale=0.21]
		\draw[fill=gray] (0.6,0) -- plot [smooth] coordinates {(0,0) (0.3,1) (0.6,0)};
		\end{scope}
		
		\begin{scope}[xshift=25,scale=0.22]
		\draw[fill=gray] (0.6,0) -- plot [smooth] coordinates {(0,0) (0.3,1) (0.6,0)};
		\end{scope}
		
		\begin{scope}[xshift=35,scale=0.23]
		\draw[fill=gray] (0.6,0) -- plot [smooth] coordinates {(0,0) (0.3,1) (0.6,0)};
		\end{scope}
		
		\begin{scope}[xshift=45,scale=0.24]
		\draw[fill=gray] (0.6,0) -- plot [smooth] coordinates {(0,0) (0.3,1) (0.6,0)};
		\end{scope}
		
		\begin{scope}[xshift=55,scale=0.25]
		\draw[fill=gray] (0.6,0) -- plot [smooth] coordinates {(0,0) (0.3,1) (0.6,0)};
		\end{scope}
		
		\begin{scope}[xshift=65,scale=0.26]
		\draw[fill=gray] (0.6,0) -- plot [smooth] coordinates {(0,0) (0.3,1) (0.6,0)};
		\end{scope}
		
		\begin{scope}[xshift=75,scale=0.27]
		\draw[fill=gray] (0.6,0) -- plot [smooth] coordinates {(0,0) (0.3,1) (0.6,0)};
		\end{scope}
		
		\begin{scope}[xshift=85,scale=0.28]
		\draw[fill=gray] (0.6,0) -- plot [smooth] coordinates {(0,0) (0.3,1) (0.6,0)};
		\end{scope}
		
		\begin{scope}[xshift=95,scale=0.29]
		\draw[fill=gray] (0.6,0) -- plot [smooth] coordinates {(0,0) (0.3,1) (0.6,0)};
		\end{scope}
		
		\begin{scope}[xshift=105,scale=0.3]
		\draw[fill=gray] (0.6,0) -- plot [smooth] coordinates {(0,0) (0.3,1) (0.6,0)};
		\end{scope}
		
		\begin{scope}[xshift=115,scale=0.31]
		\draw[fill=gray] (0.6,0) -- plot [smooth] coordinates {(0,0) (0.3,1) (0.6,0)};
		\end{scope}
		
		\begin{scope}[xshift=125,scale=0.32]
		\draw[fill=gray] (0.6,0) -- plot [smooth] coordinates {(0,0) (0.3,1) (0.6,0)};
		\end{scope}
		
		\begin{scope}[xshift=135,scale=0.33]
		\draw[fill=gray] (0.6,0) -- plot [smooth] coordinates {(0,0) (0.3,1) (0.6,0)};
		\end{scope}
		
		\begin{scope}[xshift=145,scale=0.34]
		\draw[fill=gray] (0.6,0) -- plot [smooth] coordinates {(0,0) (0.3,1) (0.6,0)};
		\end{scope}
		
		\begin{scope}[xshift=155,scale=0.35]
		\draw[fill=gray] (0.6,0) -- plot [smooth] coordinates {(0,0) (0.3,1) (0.6,0)};
		\end{scope}
		
		\begin{scope}[xshift=165,scale=0.36]
		\draw[fill=gray] (0.6,0) -- plot [smooth] coordinates {(0,0) (0.3,1) (0.6,0)};
		\end{scope}
		
		\begin{scope}[xshift=175,scale=0.37]
		\draw[fill=gray] (0.6,0) -- plot [smooth] coordinates {(0,0) (0.3,1) (0.6,0)};
		\end{scope}
		
		\begin{scope}[xshift=185,scale=0.38]
		\draw[fill=gray] (0.6,0) -- plot [smooth] coordinates {(0,0) (0.3,1) (0.6,0)};
		\end{scope}
		
		\draw[fill=gray] plot coordinates {(0,0) (7,0.05) (7,-0.05) (0,0)};
		\draw (0,0.2) --  (6.62,0.38);
		\draw[<-] plot [smooth] coordinates  {(-0.08,0.55) (-0.3,0.7) (-0.35,1.3)};
		\node at (-0.35,1.5) {\small Oval window};
		
		\begin{scope}[xshift=110]
		\draw[<-] plot [smooth] coordinates  {(2.9,0.25) (3.3,0.7) (3.35,1.4)};
		\node at (3.35,1.6) {\small Hair cells};
		\end{scope}
		
		\begin{scope}[xshift=2]
		\draw[<-] plot [smooth] coordinates  {(3.5,0.35) (3.35,0.7) (3.3,1.4)};
		\node at (3.35,1.6) {\small Tectorial membrane};
		\end{scope}
		
		\begin{scope}[xshift=-50]
		\draw[<-] plot [smooth] coordinates  {(6.8,-0.1) (6.8,-0.7) (6.7,-1.3)};
		\node at (6.5,-1.5) {\small Basilar membrane};
		\end{scope}
		
		\begin{scope}[xshift=-30]
		\draw[<-] plot [smooth] coordinates  {(2,-0.55) (2,-0.8) (2.1,-1.3)};
		\node at (2.2,-1.5) {\small Fluid-filled tube};
		\end{scope}
		\end{tikzpicture}
	\end{center}
	\caption{A cross-section of a simplified model of the cochlea.} \label{fig:diagram}
\end{figure}
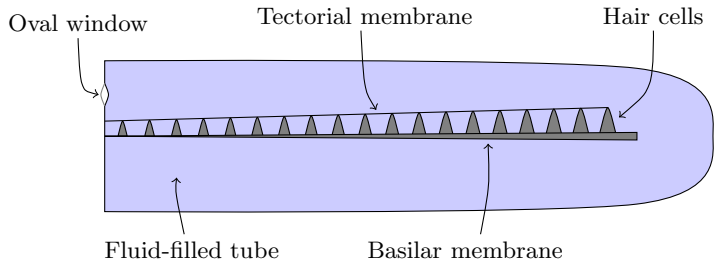

Given what is now know about the structure of the cochlea, and that its function can be reduced to studying membrane motion, one might be inclined to think that Helmholtz' resonance ideas are no longer relevant. However, the basilar membrane is much stiffer across the cochlea's width (perpendicular to the page in \Cref{fig:diagram}) than along its length. As a result, it has been shown, \emph{e.g.} by Charles Babbs in \cite{babbs2011quantitative}, that its motion can be modelled by an array of harmonic oscillators. Comparing the cochlea's length (approximately 3cm) to the wavelength of audible sound (a few centimetres to several metres) and it is clear that these resonators operate in a subwavelength regime. 

Babbs shows that if the basilar membrane is split into segments (and considered as an array of oscillators) then the interactions between each part due to elastic tension can be neglected. However, the oscillators will also be coupled by variations in the pressure of the cochlear fluid, by which they are surrounded. The mathematical complexity of modelling these interactions has been one of the main impediments facing the development of this class of cochlear models.



In this paper, we apply boundary integral techniques to understand the complex interactions between coupled subwavelength resonators \cite{ammari2017double, ammari2015mathematical}. In order to simulate the basilar membrane, we consider the problem of acoustic wave scattering by compressible elements in two-dimensional space. 
Similar layer-potential techniques have previously been applied to other materials that exhibit subwavelength resonance, the classical example being the Minnaert resonance of air bubbles in water \cite{ammari2018minnaert, ammari2017double}. This analysis (in Sections~\ref{subsec:resonant_modes}~\&~\ref{subsec:numerics}) relies on the use of layer potential techniques \cite{ammari2018mathematical, ammari2009layer, ammari2004boundary}.

It is found that a graded array of hybridised resonators has a set of resonant frequencies that becomes increasingly dense (within a finite range) as the number of resonators is increased. We study the eigenmodes and present a scheme (in \Cref{subsec:signal_processing}) for how the model processes incoming signals, filtering them into the system's resonant frequencies. Finally, in Sections~\ref{sec:cochlea}~\&~\ref{sec:tonotopic} we present the important observations that our graded-resonance model predicts the existence of a travelling wave in the pressure field and a basis for the tonotopic map. This acoustic pressure wave is complementary to the wave seen in the motion of the basilar membrane by B\'ek\'esy and has itself been observed experimentally \cite{olson1999direct}.

\section{Response of the coupled resonators} \label{sec:resonance_analysis}
\subsection{Preliminaries} \label{subsec:prelims}

We consider a domain $D$ in $\mathbb{R}^2$ which is the disjoint union of $N\in\mathbb{N}$ bounded and simply connected subdomains $\{D_1,\ldots,D_N\}$ such that, for each $n=1,\dots,N$, there is $0<s<1$ so that $\D_n\in C^{1,s}$ (that is, each $\D_n$ is locally the graph of a differentiable function whose derivatives are H\"older continuous with exponent $s$).  We will consider the resonators arranged in a straight line since the curvature of the cochlea does not contribute to its mechanical behaviour \cite{duifhuis2012cochlear}. \Cref{fig:bundles} shows an example of such an arrangement (in the special case of circular subdomains, which we will consider for the numerical simulations in \Cref{subsec:numerics} onwards).

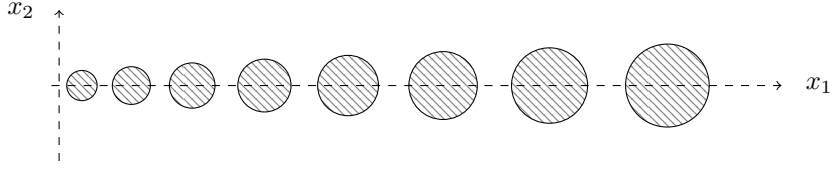
\begin{figure}
\begin{center}
\begin{tikzpicture}
\draw[pattern=north west lines, pattern color=gray] (0.3,0) circle (0.2);
\draw[pattern=north west lines, pattern color=gray] (0.95,0) circle (0.25);
\draw[pattern=north west lines, pattern color=gray] (1.75,0) circle (0.3);
\draw[pattern=north west lines, pattern color=gray] (2.7,0) circle (0.35);
\draw[pattern=north west lines, pattern color=gray] (3.8,0) circle (0.4);
\draw[pattern=north west lines, pattern color=gray] (5.05,0) circle (0.45);
\draw[pattern=north west lines, pattern color=gray] (6.45,0) circle (0.5);
\draw[pattern=north west lines, pattern color=gray] (8,0) circle (0.55);

\draw [dashed, ->] (-0.1,0) -- (9.5,0);
\draw [dashed, ->] (0,-1) -- (0,1);

\node at (-0.5,1) {$x_2$};
\node at (10,0) {$x_1$};
\end{tikzpicture}
\end{center}
\caption{An array of eight (circular) subdomains $D=D_1\cup\dots\cup D_8$ graded in size and arranged linearly along $x_2=0$.} \label{fig:bundles}
\end{figure}

We denote by $\rho_b$ and $\kappa_b$ the density and bulk modulus of the interior of the resonators, respectively, and denote by $\rho$ and $\kappa$ the corresponding parameters for the auditory fluid (which we assume occupies $\mathbb{R}^2\setminus \overline{D}$).

We consider an incident acoustic pressure wave $p^{in}(x,t)$ (where $x=(x_1,x_2)\in\mathbb{R}^2$ and $t\in\mathbb{R}$) that is scattered by $D$. This problem is given by
\begin{equation} \label{eq:wave_equation}
\begin{cases}
\left( \nabla\cdot\frac{1}{\rho}\nabla  - \frac{1}{\kappa} \ddp{^2}{t^2} \right) p = 0, & \text{for }(x,t)\in \outside\times\mathbb{R}, \\
\left( \nabla\cdot\frac{1}{\rho_b}\nabla  - \frac{1}{\kappa_b} \ddp{^2}{t^2} \right) p = 0, & \text{for }(x,t)\in D\times\mathbb{R}, \\
p_+ - p_- = 0, & \text{for }(x,t)\in\D\times\mathbb{R},\\
\frac{1}{\rho} \ddp{p}{\nu_x}\big|_+ - \frac{1}{\rho_b} \ddp{p}{\nu_x}\big|_- = 0, & \text{for }(x,t)\in\D\times\mathbb{R}, \\
p^s := p - p^{in} = 0, & \text{for } x\in\mathbb{R}^2, t \ll 0,
\end{cases}
\end{equation}
where $\ddp{}{\nu_x}$ denotes the outward normal derivative in $x$ and the  subscripts + and - are used to denote evaluation from outside and inside $\D$ respectively.

We then introduce the auxiliary parameters
\begin{equation*}
v=\sqrt{\frac{\kappa}{\rho}}, \quad v_b=\sqrt{\frac{\kappa_b}{\rho_b}}, \quad k=\frac{\omega}{v}, \quad k_b=\frac{\omega}{v_b},
\end{equation*}
which are the wave speeds and wavenumbers  in $\mathbb{R}^2\setminus \overline{D}$ and in $D$ respectively. 

We introduce the dimensionless contrast parameters
\begin{equation} \label{defn:contrasts}
\mu:=\frac{\kappa_b}{\kappa}, \quad
\delta := \frac{\rho_b}{\rho}, \quad \tau := 
\frac{v_b}{v} = \sqrt{\frac{\rho \kappa_b}{\rho_b \kappa}}.
\end{equation}
By rescaling the dimensions of the physical problem, we may assume that $v = O(1)$
and that the resonators $\{D_1,\ldots,D_N\}$ have widths that are $O(1)$. We further assume that $\tau=O(1)$. On the other hand, since each resonator $D_n$ behaves as a harmonic oscillator \cite{devaud2008minnaert}, we can use the calculations of Babbs \cite{babbs2011quantitative} to show that in order for the system of resonators to replicate the elastic properties of the basilar membrane it must be the case that 
\begin{equation}
\mu \ll 1,
\end{equation}
the details of which are given in \Cref{appen:high_contrast}. Since $\tau=\sqrt{\mu/\delta}$, these assumptions give that $\delta \ll 1$. It is important to note that this material contrast condition is an essential prerequisite for the structure $D$ to exhibit resonant behaviours at subwavelength frequencies \cite{ammari2018mathematical}.

%

We transform problem \eqref{eq:wave_equation} into the complex frequency domain by making the transformation $u(x,\omega):=\int_{-\infty}^{\infty} p(x,t) e^{i\omega t} \, dt$, $\omega\in\mathbb{C}$ to reach
\begin{equation} \label{eq:helmholtz_equation}
\begin{cases}
\left( \Delta + k^2 \right) u(x,\omega) = 0, & \text{in } \outside, \\
\left( \Delta + k_b^2 \right) u(x,\omega) = 0, & \text{in } D, \\
u_+ - u_- = 0, & \text{on } \D,\\
\delta \ddp{u}{\nu}\big|_+ -  \ddp{u}{\nu}\big|_- = 0, & \text{on } \D, \\
u^s := u - u^{in} \text{ satisfies the SRC}, & \text{as } |x|\to\infty.
\end{cases}
\end{equation}
`SRC' is used to denote the Sommerfeld radiation condition
\begin{equation} \label{eq:src}
\lim\limits_{|x|\to\infty} |x|^{1/2} \left(\ddp{}{|x|}-ik\right)u(x,\omega)=0.
\end{equation}
The SRC is the condition required to ensure that we select the solution that is outgoing (rather than incoming from infinity) and gives the well-posedness of problem \eqref{eq:helmholtz_equation}.

We wish to use integral operators known as layer potentials to represent the solution to the scattering problem \eqref{eq:helmholtz_equation}.
\begin{defn}
We define the \underline{Helmholtz single layer potential} associated with the domain $D$ and wavenumber $k$ as
\begin{equation}
\Si_{D}^k[\varphi](x) := \int_{\D} \Gamma^k(x-y) \varphi(y)\, d\sigma(y),\quad x\in\D, \varphi\in L^2(\D),
\end{equation}
where $\Gamma^k$ is the outgoing (i.e. satisfying the SRC) fundamental solution to the Helmholtz operator $\Delta+k^2$ in $\mathbb{R}^2$.
We similarly define the \underline{Neumann-Poincar\'e} operator associated with $D$ and $k$ as 
\begin{equation}
\K_{D}^{k,*}[\varphi](x)=\int_{\D} \ddp{\Gamma^k(x-y)}{\nu_x}\varphi(y)\, d\sigma(y),\quad x\in\D, \varphi\in L^2(\D).
\end{equation}
\end{defn}

We can then represent the solution to \eqref{eq:helmholtz_equation} as
\begin{equation} \label{eq:layer_potential_representation}
u = \begin{cases}
u^{in}(x)+\Si_{D}^k[\psi](x), & x\in\mathbb{R}^2\backslash \overline{D},\\
\Si_{D}^{k_b}[\phi](x), & x\in D,
\end{cases}
\end{equation} for some surface potentials $(\phi,\psi)\in L^2(\D)\times L^2(\D)$.

We define the space $H^1(\D):=\{u\in L^2(\D): \nabla u\in L^2(\D)\}$ in the usual way and use $Id$ to denote the identity on $L^2(\D)$. Then, using the representation \eqref{eq:layer_potential_representation}, problem \eqref{eq:helmholtz_equation} is equivalent \cite{ammari2004boundary, ammari2009layer} to finding $(\phi,\psi)\in L^2(\D)\times L^2(\D)$ such that
\begin{equation} \label{eq:A_matrix_equation}
\A(\omega,\delta)\begin{pmatrix} \phi \\ \psi \end{pmatrix} = \begin{pmatrix} u^{in} \\ \delta\ddp{u^{in}}{\nu_x}\end{pmatrix},
\end{equation}
where
\begin{equation} \label{eq:A_matrix_defn}
\A(\omega,\delta):=
\left[ {\begin{array}{cc}
	\Si_D^{k_b} & -\Si_D^k \\
	-\frac{1}{2}Id+\K_D^{k_b,*} & -\delta(\frac{1}{2}Id+\K_D^{k,*}) \\
	\end{array} } \right].
\end{equation}

We now recall from \emph{e.g.} \cite{ammari2018minnaert, ammari2018mathematical} the main result that will allow us to understand the leading order behaviour of $\A$ in \eqref{eq:A_matrix_equation}.

\begin{lemma} \label{lem:expansion}
	In the space $\mathcal{L}(L^2(\D)\times L^2(\D),H^1(\D)\times L^2(\D))$ we have
	
	\begin{equation*}
	\A(\omega,\delta)=\A_0+\omega^2\ln\omega\A_{1,1,0}+\omega^2\A_{1,2,0}+\delta\A_{0,1}+O(\delta\omega^2\ln\omega)+O(\omega^4\ln\omega),
	\end{equation*}
	where
	\begin{equation*}
	\A_0:=
	\left[ {\begin{array}{cc}
		\hat{\Si}_D^{k_b} & -\hat{\Si}_D^k \\
		-\frac{1}{2}Id+\K_D^{*} & 0 \\
		\end{array} } \right],
	\A_{1,1,0}:=
	\left[ {\begin{array}{cc}
	v_b^{-2}\Si_{D,1}^{(1)} & -v^{-2}\Si_{D,1}^{(1)} \\
	v_b^{-2}\K_{D,1}^{(1)} & 0 \\
	\end{array} } \right],
	\end{equation*}
	\begin{equation*}
	\A_{1,2,0}:=
	\left[ {\begin{array}{cc}
		v_b^{-2}(-\ln v_b\Si_{D,1}^{(1)}+\Si_{D,1}^{(2)}) &
		-v^{-2}(-\ln v\Si_{D,1}^{(1)}+\Si_{D,1}^{(2)}) \\
		v_b^{-2}(-\ln v_b\K_{D,1}^{(1)}+\K_{D,1}^{(2)}) & 0 \\
		\end{array} } \right],
	\end{equation*}
	and
	\begin{equation*}
	\A_{0,1}:=
	\left[ {\begin{array}{cc}
		0 & 0 \\
		0 & -(\frac{1}{2}Id+\K_D^{*}) \\
		\end{array} } \right].
	\end{equation*}
	The above operators are defined as
	\begin{align*}
	\Si_D[\phi](x)&:=\frac{1}{2\pi}\int_{\D}\ln|x-y|\phi(y)\,d\sigma(y),\\
	\hat{\Si}_D^k[\phi](x)&:=\Si_D[\phi](x)+\eta_k\int_{\D}\phi \,d\sigma,\quad
	\eta_k:=\frac{1}{2\pi}(\ln k+\gamma-\ln 2)-\frac{i}{4},\\
	\Si_{D,1}^{(1)}[\phi](x)&:= \int_{\D} b_1|x-y|^2\phi(y)\, d\sigma(y),\\
	\Si_{D,1}^{(2)}[\phi](x)&:= \int_{\D} b_1|x-y|^2\ln|x-y|\phi(y) +c_1|x-y|^2\phi(y)\, d\sigma(y),\\
	\K_{D,1}^{(1)}[\phi](x)&:=\int_{\D} b_1 \frac{\partial |x-y|^2}{\partial \nu(x)} \phi(y)\, d\sigma(y),\\
	\K_{D,1}^{(2)}[\phi](x)&:=\int_{\D} b_1 \frac{\partial |x-y|^2\ln|x-y|}{\partial \nu(x)} \phi(y)+c_1 \frac{\partial |x-y|^2}{\partial \nu(x)} \phi(y)\, d\sigma(y),
	\end{align*}
	where $b_1:=-\frac{1}{8\pi}$, $c_1:=-\frac{1}{8\pi}(\gamma-\ln 2-1-\frac{i\pi}{2})$ and $\gamma=0.5772\ldots$ is the Euler constant.
\end{lemma}

The operator $\Si_{D}$ is the Laplace single layer potential associated with $D$. Since we are working in two dimensions this is not generally invertible however the following two lemmas help us understand the extent of its degeneracy.
\begin{lemma} \label{eq:l0_inj}
	If for some $\phi\in L^2(\D)$ with $\int_{\D}\phi=0$ it holds that $\Si_D[\phi](x)=0$ for all $x\in\D$, then $\phi=0$ on $\D$.
\end{lemma}
\begin{proof}
	The arguments given in \cite[Lemma 2.25]{ammari2007polarization} can be easily generalised to the case where $D$ is the disjoint union of a finite number of bounded Lipschitz domains in $\mathbb{R}^2$.
\end{proof}

\begin{prop} \label{dim_ker_Si}
	Independent of the number $N\in\mathbb{N}$ of connected components making up $D$, we have that 
	\begin{equation*}
	\dim\ker\Si_D\leq1.
	\end{equation*}
\end{prop}

\begin{proof}
	Let $\psi\in\ker\Si_D$. Then by Lemma \ref{eq:l0_inj} if $\int_{\D}\psi=0$ then $\psi=0$. Suppose that $\int_{\D}\psi\neq0$ then take $\tilde{\psi}\in\ker\Si_D$ with $\int_{\D}\tilde{\psi}\neq0$ and then consider the function
	\begin{equation*}
	f=\frac{\psi}{\int_{\D}\psi}-\frac{\tilde{\psi}}{\int_{\D}\tilde{\psi}}.
	\end{equation*}
	Then $f$ satisfies $\Si_D[f]=0$ and $\int_{\D}f=0$ so by Lemma  \ref{eq:l0_inj} we have that $f=0$. Therefore $\psi=(\int_{\D}\psi/\int_{\D}\tilde{\psi})\tilde{\psi}$.
\end{proof}

There are two cases to consider, in light of \Cref{dim_ker_Si}:
\begin{itemize}
	\itemsep0em 
	\item Case I: $\dim\ker\Si_D=1$,
	\item Case II: $\dim\ker\Si_D=0$.
\end{itemize}
\noindent By the Fredholm Alternative Theorem, an equivalent formulation is:
\begin{itemize}
	\itemsep0em 
	\item Case I: $\Si_D$ is not invertible,
	\item Case II: $\Si_D$ is invertible,
\end{itemize}
\noindent as an operator in $\mathcal{L}(L^2(\D),H^1(\D))$. We are now in a position to prove an important property of the operator $\hat{\Si}_D^k$ that was defined in \Cref{lem:expansion} and is the leading order approximation to $\Si_D^k$ as $k\to0$.

\begin{lemma} \label{thm:skD_invertible}
	For any fixed $k\in\mathbb{C}\setminus\{z\in\mathbb{C}:\Re(z)=0,\Im(z)\geq0\}$, $\hat{\Si}_D^k$ is invertible in $\mathcal{L}(L^2(\D),H^1(\D))$.
\end{lemma}
\begin{proof}
	Since $\hat{\Si}_D^k$ is Fredholm with index 0 we need only to show that it is injective. To this end, assume that $y\in L^2(\D)$ is such that
	\begin{equation} \label{eq:Shat_inj}
	\hat{\Si}_D^k[y]=\Si_D[y]+\eta_k\int_{\D}y=0.
	\end{equation}

	\noindent	\underline{\text{Case I}:}
	Let $\psi_0$ be the unique element of $\ker\Si_D$ with $\int_{\D}\psi_0=1$ (which exists as a result of \Cref{eq:l0_inj}). We then find that $\Si_D[y]\perp\psi_0$ in $L^2(\D)$ and hence \eqref{eq:Shat_inj} becomes
	\begin{equation*}
	\eta_k\left(\int_{\D}y\right)\left(\int_{\D}\psi_0\right)=0.
	\end{equation*} Thus $\int_{\D}y=0$. It follows from \eqref{eq:Shat_inj} that $\Si_D[y]=0$ and further by Lemma \ref{eq:l0_inj} we have that $y=0$.
	
	\noindent	\underline{\text{Case II}:}
	Define $\psi_0=\Si_D^{-1}(1)$ then \eqref{eq:Shat_inj} gives us that
	\begin{equation*}
	\Si_D[y]=-\eta_k\int_{\D}y,
	\end{equation*}
	is constant so, since $\Si_D$ is injective, we find that $y=c\psi_0$ for some $c$.
	Substituting back into \eqref{eq:Shat_inj} gives 
	\begin{equation*}
	c\left(1+\eta_{k}\int_{\D}\psi_0\right)=0.
	\end{equation*}
	Everything within the brackets is real with the one exception of $\eta_k$ (which has nonzero imaginary part, thanks to the choice of $k$) so we must have that $c=0$.
\end{proof}

\subsection{Resonant modes} \label{subsec:resonant_modes}

\begin{defn} \label{defn:resonance}
	For a fixed $\delta$ we define a \underline{resonant frequency} to be $\omega\in\mathbb{C}$ with positive real part and negative imaginary part such that there exists a nontrivial solution to
\begin{equation} \label{eq:res}
\A(\omega,\delta)
\begin{pmatrix}
\phi \\ \psi
\end{pmatrix}
=
\begin{pmatrix}
0 \\ 0
\end{pmatrix},
\end{equation}
where $\A(\omega,\delta)$ is defined in \eqref{eq:A_matrix_defn}.
For each resonant frequency $\omega$ we define the corresponding \underline{eigenmode} (or \underline{resonant mode} or \underline{normal mode}) as
\begin{equation} 
u = \begin{cases}
\Si_{D}^k[\psi](x), & x\in\mathbb{R}^2\backslash D,\\
\Si_{D}^{k_b}[\phi](x), & x\in D.
\end{cases}
\end{equation}
\end{defn}

\begin{remark}
	The reason for the choices of sign in \Cref{defn:resonance} is to give a physical meaning to a complex resonant frequency. The real part represents the frequency of oscillation and the imaginary part describes the rate of attenuation (hence it should be negative, to give a solution that decays over time).
\end{remark}
\begin{remark}
	We will see from \Cref{fig:frequency_response_N=6} that \Cref{defn:resonance} is equivalent to the notion that resonant frequencies are those at which the system will oscillate at much greater amplitude than is generally the case.
\end{remark}

We wish to now compute the resonant frequencies and associated eigenmodes for our system. Manipulating the first entry of \eqref{eq:res} we find that
\begin{equation*}
\hat{\Si}_D^{k_b}[\phi]-\hat{\Si}_D^{k}[\psi]
=\hat{\Si}_D^{k}[\phi-\psi]+\frac{1}{2\pi}\ln \frac{v}{v_b}\int_{\D}\phi,
\end{equation*}
hence
\begin{equation} \label{eq:phi_expression}
\psi = \phi+\frac{1}{2\pi}\ln \frac{v}{v_b}\left(\int_{\D}\phi\right) (\hat{\Si}_D^k)^{-1}[\chi_{\D}]+O(\omega^2),
\end{equation}
since an application of $(\hat{\Si}_D^k)^{-1}$ rescales like $O(1/\ln\omega)$. Here, $\chi_{\D}$ is used to denote the characteristic function of $\D$.

To deal with the second component of \eqref{eq:res} we first prove some technical lemmas.

\begin{lemma} \label{lem:K1}
	For any $\phi\in L^2(\D)$ and $j=1,\ldots,N$, we have that
	
	(i) $\int_{\D_j}(\frac{1}{2}I-\K_D^*)[\phi] = 0$,
	
	(ii) $\int_{\D_j}(\frac{1}{2}I+\K_D^*)[\phi] = \int_{\D_j}\phi$.
\end{lemma}
\begin{proof}
	(i) follows from the jump relations for single layer potentials and the fact $\Si_D[\phi]$ is harmonic in $D$ \cite{ammari2009layer, ammari2007polarization}. Then (ii) is immediate.
\end{proof}

\begin{lemma} \label{lem:K2}
	For any $\phi\in L^2(\D)$ and $j=1,\ldots,N$, we have that
	
	(i) $\int_{\D_j} \K_{D,1}^{(1)}[\phi]=4b_1|D_j|\int_{\D} \phi$,
	
	(ii) $\int_{\D_j}\K_{D,1}^{(2)}[\phi] = -\int_{D_j} \Si_D[\phi] + (4b_1+4c_1)|D_j| \int_{\D}\phi $,
	
	where $|D_i|$ is the area of $D_i$.
\end{lemma}

\begin{proof}
	(i) follows from the divergence theorem 
	\begin{align*}
	\int_{\D_j} \K_{D,1}^{(1)}[\phi](x)\,d\sigma(x)
	&= b_1 \int_{D_j}\int_{\D}\Delta_x|x-y|^2\phi(y) \,d\sigma(y)\,dx \\
	&= 4b_1 |D_j| \int_{\D}\phi(y)\,d\sigma(y).
	\end{align*}
	
	Similarly for (ii) we can show that
	\begin{align*}
	\int_{\D_j}K_{D,1}^{(2)}[\phi](x)\,d\sigma(x)
	&= \int_{D_j}\int_{\D}\Delta_x [|x-y|^2(b_1 \ln |x-y|+c_1) ]\phi(y) \,d\sigma(y)\,dx \\
	&= -\int_{D_j} \Si_D[\phi](x)\,dx + (4b_1+4c_1)|D_j| \int_{\D}\phi(y) \,d\sigma(y),
	\end{align*}
	making use of the fact that $b_1=-1/8\pi$.
\end{proof}

Turning now to the second component of \eqref{eq:res} we see that
\begin{equation*}
\begin{split}
\left(-\frac{1}{2}Id+\K_D^{*} + v_b^{-2}\K_{D,1}^{(1)}\omega^2\ln\omega + v_b^{-2}(-\ln v_b\K_{D,1}^{(1)}+\Si_{D,1}^{(2)})\omega^2\right)[\phi]
\\-\delta(\frac{1}{2}Id+\K_D^{*})[\psi]=O(\delta\omega^2\ln\omega)+O(\omega^4\ln\omega).
\end{split}
\end{equation*}
We substitute expression \eqref{eq:phi_expression} for $\psi$ to see that $\phi$ satisfies the equation
\begin{equation} \label{eq:phi}
\begin{split}
\left(-\frac{1}{2}Id+\K_D^{*}\right)[\phi] +
\left( v_b^{-2}\K_{D,1}^{(1)}\omega^2\ln\omega + v_b^{-2}(-\ln v_b\K_{D,1}^{(1)}+\K_{D,1}^{(2)})\omega^2\right)[\phi]
\\-\delta(\frac{1}{2}Id+\K_D^{*})[\phi]-\frac{1}{2\pi}\delta \ln \frac{v}{v_b}\left(\int_{\D}\phi\right) \left(\frac{1}{2}Id+\K_D^{*}\right)\left[ (\hat{\Si}_D^k)^{-1}[\chi_{\D}]\right]\\
=O(\delta\omega^2\ln\omega)+O(\omega^4\ln\omega).
\end{split}
\end{equation}

At leading order \eqref{eq:phi} is just $(-\frac{1}{2}Id+\K_D^{*})[\phi]=0$ so it would be useful to understand this kernel, which we achieve with the following two lemmas.

\begin{lemma} \label{ker_Si_constant}
	If $\phi\in L^2(\D)$ is such that $\phi\in\ker(-\frac{1}{2}Id+\K_D^*)$ then there exist constants $b_j$ 
	such that $\Si_D[\phi]=\sum_{j=1}^N b_{j}\mathcal{X}_{\D_j}$.
\end{lemma}

\begin{proof}
	Let $u:=\Si_D[\phi]$. Then $\Delta u=0$ in $D$ and $\ddp{u}{\nu}\big|_-=(-\frac{1}{2}Id+\K_D^*)[\phi]=0$ on $\D$ (known as a ``jump condition'') \cite{ammari2004boundary,ammari2009layer} so $u$ satisfies a homogeneous interior Neumann problem on each of the $N$ connected components $D_1,\dots,D_N$ of $D$. It is known that such problems are uniquely solvable up to the addition of a constant.
\end{proof}

\begin{lemma} \label{kernel_Ndim}
	Fix some ${k_0}\in\mathbb{C}\setminus\{0\}$. The set of vectors $\{\psi_1,\ldots,\psi_N\}$ defined as
	\begin{equation} \label{defn:basis_vectors}
	\psi_i:=\left(\hat{\Si}_D^{k_0}\right)^{-1}[\mathcal{X}_{\D_i}],
	\end{equation}
	forms a basis for the space $\ker(-\frac{1}{2}Id+\K_D^*)$.
\end{lemma}

\begin{proof}
	The linear independence of $\{\psi_1,\ldots,\psi_N\}$ follows from the linearity and injectivity of $\hat{\Si}_D^{k_0}$, plus the independence of $\{\mathcal{X}_{\D_1},\ldots,\mathcal{X}_{\D_N}\}$.
	
	For $\phi\in L^2(\D)$ the difference between $\hat{\Si}_D^{k_0}[\phi](x)$ and $\Si_D[\phi](x)$ is a constant (in $x$) so they will have the same derivatives. In particular, they are both harmonic and satisfy the same jump conditions across $\D$. Therefore, using arguments as in \Cref{ker_Si_constant}, we see that if $\phi\in\ker(-\frac{1}{2}Id+\K_D^*)$ then $\hat{\Si}_D^{k_0}[\phi]\in\spn\{\mathcal{X}_{\D_1},\ldots,\mathcal{X}_{\D_N}\}$. Thus $\phi\in\spn\{\psi_1,\ldots,\psi_N\}$.
\end{proof}

From \Cref{kernel_Ndim} we know that $\ker(-\frac{1}{2}Id+\K_D^*)$ has dimension equal to the number of connected components of $D$ (a wider discussion can be found in \emph{e.g.} \cite{ammari2013spectral}). Thus we can take a basis
\begin{equation*}
\{\phi_1,\ldots,\phi_N\},
\end{equation*}
of the null space $\ker(-\frac{1}{2}Id+\K_D^*)$. 
Then, in light of the fact that at leading order \eqref{eq:phi} is just $(-\frac{1}{2}Id+\K_D^{*})[\phi]=0$, it is natural to seek a solution of the form
\begin{equation} \label{eq:phi_is_span}
\phi=\sum_{j=1}^{N}a_j\phi_j+O(\omega^2\ln\omega+\delta),
\end{equation}
for some non-trivial constants $a_j$ with $\sum_j|a_j|=O(1)$. The solutions $(\phi,\psi)$ to \eqref{eq:res} are determined only up to multiplication by a constant (and hence so are $a_1,\dots,a_N$). We fix the scaling to be such that the eigenmodes are normalised in the $L^2(D)$-norm
\begin{equation}
\|u\|_{L^2(D)}^2 = \int_{D} |\Si_{D}^{k_b}[\phi]|^2 = 1.
\end{equation}

We now integrate \eqref{eq:phi} over each $\D_i, i=1\ldots N$ and use the results of Lemmas \ref{lem:K1} and \ref{lem:K2} to find that, up to an error of $O(\delta\omega^2\ln\omega)+O(\omega^4\ln\omega)$,

\begin{equation} \label{eq:Bi_defn}
\begin{split}
B_\delta^{(i)}(\omega)[\phi]:=\left(\int_{\D}\phi\right) \left(\omega^2\ln\omega
+\left(\left(1+\frac{c_1}{b_1}-\ln v_b\right)
-\frac{\Si_D[\phi]|_{\D_i}}{4b_1( \int_{\D}\phi )} \right)\omega^2 \right)\\ -\frac{v_b^2}{4b_1|D_i|}\left[\int_{\D_i}\phi
+\frac{\ln (v/v_b)}{2\pi} \left(\int_{\D}\phi\right)\int_{\D_i} (\hat{\Si}_D^k)^{-1}[\chi_{\D}]\right]\delta=0.
\end{split}
\end{equation}

When we substitute the expression \eqref{eq:phi_is_span} for $\phi$ in \eqref{eq:Bi_defn} we find the system of equations, up to an error of order $O(\delta\omega^2\ln\omega)+O(\omega^4\ln\omega)$,

\begin{equation} \label{eq:Bmatrix}
\begin{pmatrix}
B_\delta^{(1)}(\omega)[\phi_1] & B_\delta^{(1)}(\omega)[\phi_2] & \dots & B_\delta^{(1)}(\omega)[\phi_N] \\
\vdots & \vdots & \ddots & \vdots \\
B_\delta^{(N)}(\omega)[\phi_1] & B_\delta^{(N)}(\omega)[\phi_2] & \dots & B_\delta^{(N)}(\omega)[\phi_N]
\end{pmatrix}
\begin{pmatrix}
a_1 \\ \vdots \\ a_N
\end{pmatrix}
=0.
\end{equation}

\begin{remark} \label{rk:basis_indep}
	Thanks to the linearity of the operators $B_\delta^{(i)}$, the solutions $\omega(\delta)$ to \eqref{eq:Bmatrix} (as well as the associated eigenmodes) are independent of the choice of basis $\{\phi_1,\dots,\phi_N\}$.
\end{remark}

\begin{remark}
	One can think of the step where we integrated \eqref{eq:phi} over each $\D_i$, for $i=1\dots,N$, to give \eqref{eq:Bi_defn} as the point where the hybridisation (between the $N$ resonators) was performed (see also e.g. \cite{ammari2017double}).
\end{remark}

\subsection{Numerical computations of resonant modes} \label{subsec:numerics}

In order to improve computational efficiency, we will assume from here onward that the resonators are circular. This means that we can use the so-called multipole expansion method, an explanation of which is provided in \emph{e.g.}~\cite[Appendix C]{ammari2017subwavelength}. The method relies on the idea that functions in $L^2(\D)$ are, on each \textit{circular} $\D_i$, $2\pi$-periodic so we may approximate by the leading order terms of a Fourier series representation. We found that as few as seven terms was sufficient to give satisfactory results.

Using such an approach we can find, for each fixed $\delta>0$, the $N$ values of $\omega\in\mathbb{C}$ such that there exists a nontrivial solution to \eqref{eq:Bmatrix}. For the case where $N=50$ the results are shown in \Cref{fig:resonances_example}.
\begin{figure}
	\begin{center}
		\includegraphics[width=\textwidth]{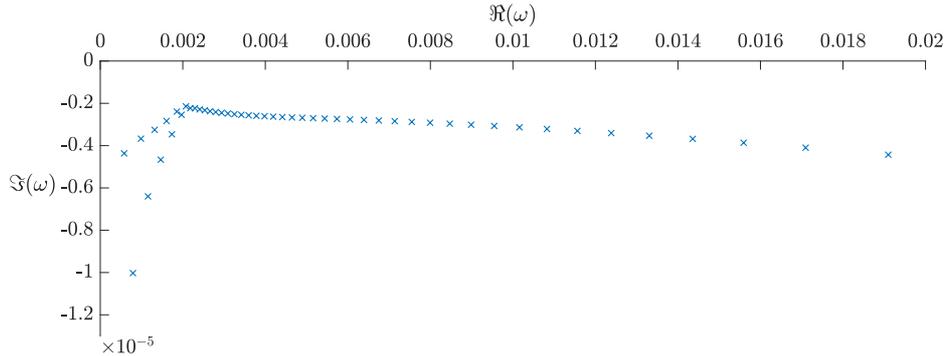}
		\caption{The resonant frequencies, plotted in the complex plane, of a system of 50 resonators arranged linearly with each being 1.05 times the size of the previous.
		The first resonance $\omega_1=0.0002284-0.0000526i$ is omitted. We take $\delta=1/7000$ in this simulation.} \label{fig:resonances_example}
	\end{center}
\end{figure}
We see that there is a range of frequencies where the (the real part of the) resonances occur most commonly. As $N$ is increased, the resonances become increasingly dense in this region. In fact, with the current arrangement, this range of frequencies does not change as $N$ increases. Instead, the region becomes increasingly densely filled.

It is also seen from \Cref{fig:resonances_example} that the imaginary parts of the resonances is smallest in the region where they are most dense. This means that these frequencies experience the least significant attenuation, suggesting that tones in this range will be most easily audible. The reason $\omega_1=0.0002284-0.0000526i$ has been omitted from \Cref{fig:resonances_example} is due to its $O(10^{-4})$ imaginary part. This is not only inconvenient for plotting but also means that this resonant mode will suffer much greater attenuation and thus will be a less significant part of the motion over time.

\begin{figure}
	\begin{center}
		\makebox[\textwidth][c]{\includegraphics[width=1.3\textwidth,trim={0.5cm 3cm 0.5cm 1cm},clip]{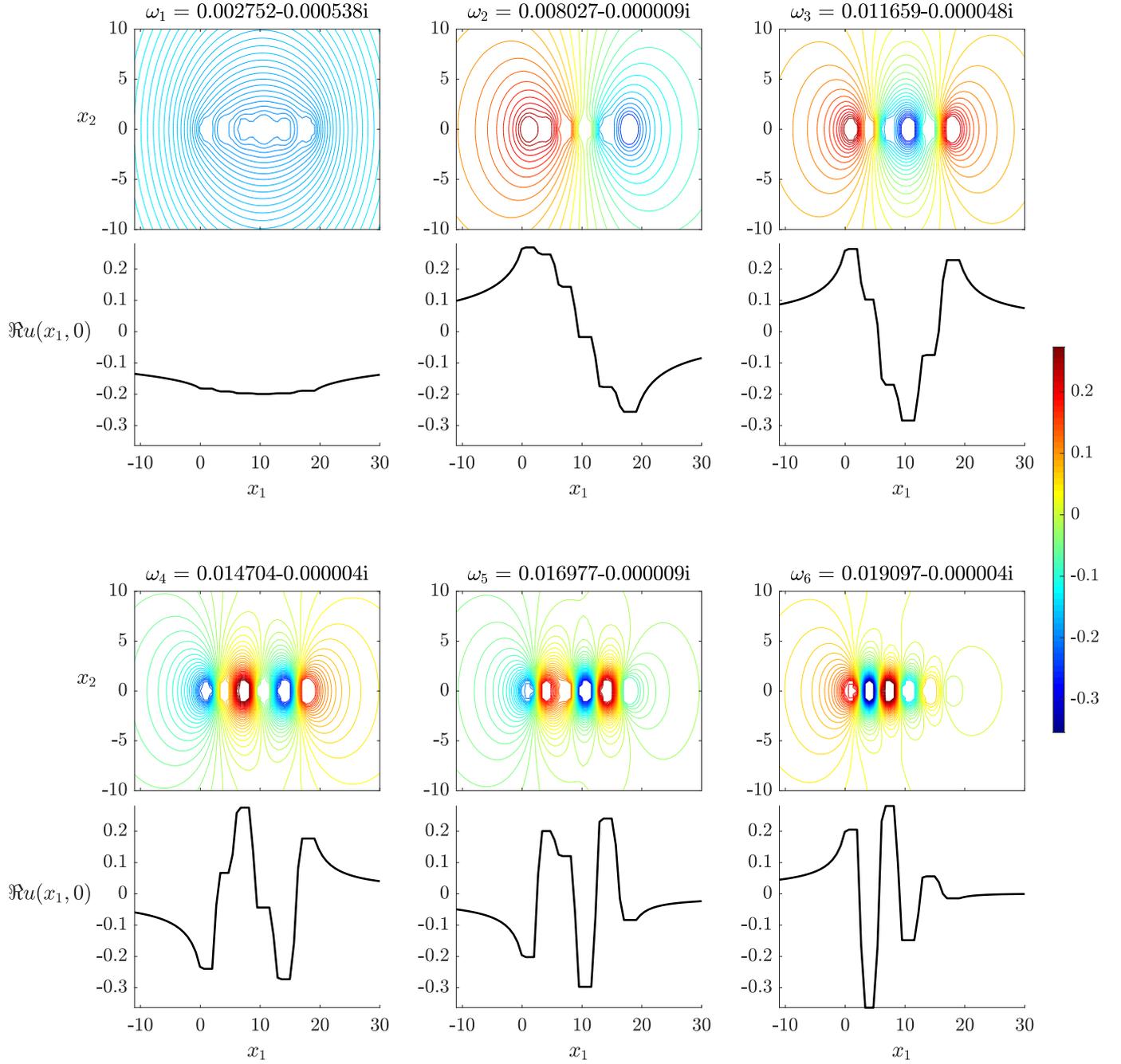}}
		\caption{The acoustic pressure eigenmodes $u_1,\dots,u_6$ for a system of six resonators arranged linearly with each being 1.05 times the size of the previous (smallest on the left). Each pair of plots corresponds to one of the six resonant frequencies. The upper plot shows a contour plot of the function $\Re u_n(x_1,x_2)$. The lower plot shows the cross section of this, taken along the line $x_2=0$ (\emph{i.e.} through the centres of the resonators). The eigenmodes have been normalised such that $\int_D |u_n|^2\,dx=1$ for each $n=1,\dots,N$.
		} \label{fig:eigenmodes_example}
	\end{center}
\end{figure}

It is also important to understand the eigenmodes $u_n$ associated with each resonant frequency $\omega_n$. The six resonant modes for the case of six resonators are shown in \Cref{fig:eigenmodes_example}. They take the form of increasingly oscillating patterns that inherit the asymmetry of the resonator array. 

It is also notable that the solution is approximately constant on each resonator. This is because the solution, taking the form \eqref{eq:layer_potential_representation}, is given by $\hat{\Si}_D^{k_b}[\phi]$ at leading order which by \Cref{ker_Si_constant} is constant for $\phi\in\ker(-\frac{1}{2}Id+\K_D^*)$.


\subsection{Signal processing} \label{subsec:signal_processing}

We wish to offer an explanation of how, given an incident wave $p^{in}(x,t)$, our system of coupled resonators is able to classify (and hence identify) the sound. The system of resonators $D$ is able to decompose the signal over its resonant modes. It is clear that the $N$ eigenmodes are linearly independent so we may define the relevant $N$-dimensional solution spaces.
\begin{defn} We define the $N$-dimensional spaces $X$ and $Y$ as
	\begin{equation}
	X:=\spn\{ u_1(x),\ldots,u_N(x)\},
	\end{equation}
	\begin{equation}
	Y:=\spn\{u_1(x)e^{-i\omega_1t},\dots,u_N(x)e^{-i\omega_Nt}\},
	\end{equation}
\end{defn}

We will approximate the solution by a decomposition in the frequency domain.
The fact that, for $n=1,\dots,N$, the Fourier transform of $e^{-i\omega_n t}$ for $t>0$ is given by $i/(\omega-\omega_n)$ motivates us to employ the ansatz
\begin{equation} \label{eq:decomposition_freq}
u(x,\omega)\simeq\sum_{n=1}^{N} \frac{\alpha_n(\omega)i}{\omega-\omega_n} u_n(x),
\end{equation}
where $\alpha_1,\dots,\alpha_N$ are complex-valued functions of a real variable.

It is important to understand whether knowing the value of the solution on each resonator (which is the information that a cochlea is able to capture) means that one can recover the weight functions $\alpha_1,\dots,\alpha_N$ in \eqref{eq:decomposition_freq}. 

\begin{remark}
	The eigenmodes $u_1,\dots,u_N$ are not orthogonal in $L^2(D)$. It turns out, however, that they are \emph{nearly} orthogonal. For example, the normalised eigenmodes shown in \Cref{fig:eigenmodes_example} satisfy $(u_n,u_m)_{L^2(D)}=O(10^{-3})$ for $n\neq m$.
\end{remark}

\begin{prop} \label{prop:gamma_invertible}
	Let $\{\omega_1,\ldots,\omega_N\}$ be the resonances of the system $D=D_1\cup\ldots\cup D_N$ and denote by $u_1,\ldots,u_N$ the corresponding eigenmodes. Then the matrix $\gamma\in\mathbb{C}^{N\times N}$ defined by
	\begin{equation}
	\gamma_{ij}:=\int_D u_i(x) \overline{u_j(x)}\, dx \quad i,j=1\ldots N,
	\end{equation}
	is invertible.
\end{prop}

\begin{proof}
	We can apply the Gram-Schmidt procedure to produce a basis $\{ v_1,\ldots,v_N\}$ for $X$ that is orthonormal with respect to $(\cdot,\cdot)_{L^2(D)}$. This procedure produces a nonsingular lower triangular matrix $P\in\mathbb{C}^{N\times N}$ such that $(v_1,\dots,v_N)^T=P(u_1,\dots,u_N)^T$ (superscript $T$ denotes the matrix transpose). If we define $Q\in\mathbb{C}^{N\times N}$ as $Q:=P^{-1}$ then $Q$ is also nonsingular and lower triangular. We can then calculate that
	\begin{equation} \label{eq:gram-s_matrix}
	\begin{bmatrix}
	u_1 & \dots & u_N
	\end{bmatrix}^T
	\begin{bmatrix}
	\overline{u_1} & \dots & \overline{u_N}
	\end{bmatrix}
	= Q
	\begin{bmatrix}
	v_1 & \dots & v_N
	\end{bmatrix}^T
	\begin{bmatrix}
	\overline{v_1} & \dots & \overline{v_N}
	\end{bmatrix} \overline{Q}^T.
	\end{equation}
	Integrating \eqref{eq:gram-s_matrix} componentwise gives that, for $i,j=1,\dots, N$, it holds that
	\begin{equation}
	\gamma_{ij} = \left[QI_N\overline{Q}^T\right]_{ij},
	\end{equation}
	and thus
	\begin{equation}
	\det(\gamma) = |\det(Q)|^2>0.
	\end{equation}
	
\end{proof}

In order to find the weight functions $\alpha_1,\dots,\alpha_N$ in \Cref{eq:decomposition_freq} we must take the $L^2(D)$-product with $u_n(x)$ for $n=1,\dots,N$ and then invert $\gamma$. This gives that
\begin{equation} \label{eq:solve_for_alpha}
\begin{pmatrix}
\frac{\alpha_1(\omega)i}{\omega-\omega_1} \\ \vdots \\ \frac{\alpha_N(\omega)i}{\omega-\omega_N}
\end{pmatrix}
= \overline{\gamma}^{-1}
\begin{pmatrix}
\left(u(\cdot,\omega),u_1\right)_{L^2(D)} \\ \vdots \\ \left(u(\cdot,\omega),u_N\right)_{L^2(D)}
\end{pmatrix}.
\end{equation}
{
Thanks to its representation \eqref{eq:layer_potential_representation} in terms of single layer potentials, $u(\cdot,\omega)$ is an analytic function of $\omega\in\mathbb{C}$. Thus, from \eqref{eq:solve_for_alpha} we can see that $\alpha_1,\dots,\alpha_N$ are analytic and hence we can recover a similar decomposition for $p(x,t)$ using the Laplace inversion theorem
}

\begin{equation} \label{eq:decomposition_time}
\begin{split}
p(x,t) &\simeq \frac{1}{2\pi}\sum_{n=1}^{N} u_n(x) \int_{-\infty}^{\infty} \frac{\alpha_n(\omega)i}{\omega-\omega_n}e^{-i\omega t} \,d\omega \\
&= \sum_{n=1}^{N} u_n(x) \alpha_n(\omega_n) e^{-i\omega_nt}, \quad t>0.
\end{split}
\end{equation}

\begin{figure}
	\begin{center}
		\includegraphics[width=\textwidth,trim={0 0.5cm 0 2cm}]{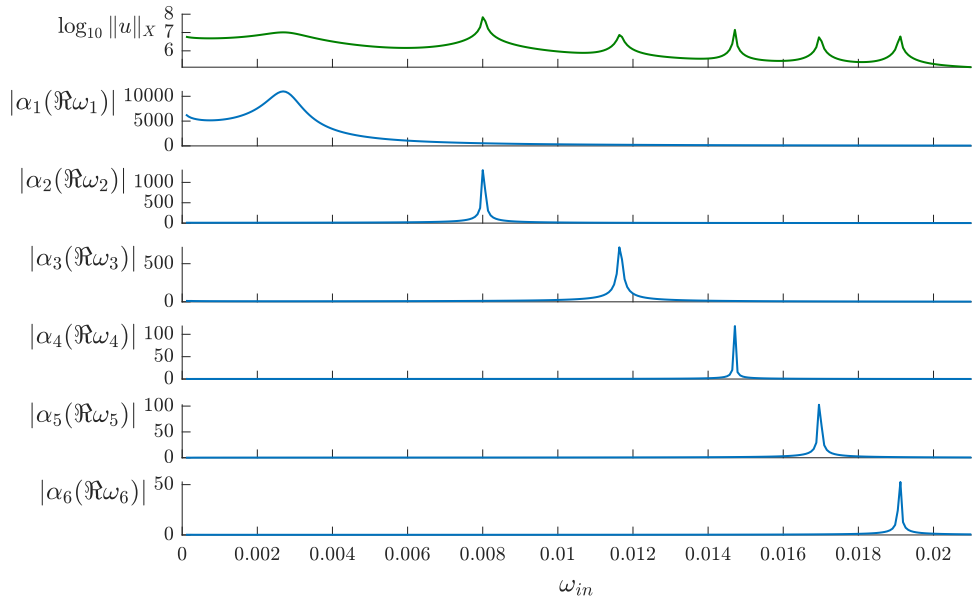}
	\end{center}
	\caption{A system of six resonators filters an acoustic signal into the six resonant frequencies. We consider a system of six linearly arranged circular resonators that increase in size by a factor of $1.05$ which is subjected to an incoming plane wave with frequency $\omega_{in}$. 
		The first plot shows how the norm of the solution $u(x,\omega)$ to \eqref{eq:helmholtz_equation} varies as a function of $\omega_{in}$. 
		We then show how each coefficient $\alpha_1(\omega_1),\dots,\alpha_N(\omega_N)$ in \eqref{eq:decomposition_time} varies.
		The six resonant frequencies of this system are $\omega_1= 0.002752 -0.000538i$,
		$\omega_2=   0.008026 -0.000009i $,
		$\omega_3=   0.011659 -0.000048i $,
		$\omega_4=   0.014703 -0.000004i $,
		$\omega_5=   0.016976 -0.000009i $,
		$\omega_6=   0.019096 -0.000004i $.
	}
	\label{fig:frequency_response_N=6}
\end{figure}

\begin{example}
	$p^{in}(x,t)$ is a plane wave
\end{example}
\noindent We take as an example the case where $p^{in}(x,t)$ is a pulse of a plane wave with frequency $\omega_{in}\in\mathbb{R}$ travelling in the $x_1$ direction. This is given by
\begin{equation} \label{eq:plane_wave}
p^{in}(x,t)=e^{i\omega_{in}(x_1/v-t)}, \quad 0<t<1.
\end{equation}
This has Fourier transform
\begin{equation}
u^{in}(x,\omega)=2e^{\frac{i}{2}(\omega-\omega_{in})}\sinc(\omega-\omega_{in})e^{i\omega_{in} x_1/v}.
\end{equation}
We can then compute $\alpha_1(\omega),\dots,\alpha_N(\omega)$ as in \eqref{eq:solve_for_alpha}.

In \Cref{fig:frequency_response_N=6} we show firstly how the $L^2(D)$-norm of the solution to the scattering problem \eqref{eq:helmholtz_equation} varies as a function of $\omega_{in}$. As is expected, the response is (locally) much greater when $\omega_{in}$ is close to $\Re(\omega_n)$ for some $n=1,\dots,N$. We also show how the weights $\alpha_1(\omega_1),\dots,\alpha_N(\omega_N)$ in \eqref{eq:decomposition_time} vary as a function of $\omega_{in}$. Each constant is small except in a region of the associated resonant frequency when the corresponding eigenmode is excited most strongly. 

\begin{wrapfigure}{R}{0.45\textwidth}
	\centering
	\includegraphics[width=0.45\textwidth,trim={0.5cm 0 0.5cm 0}]{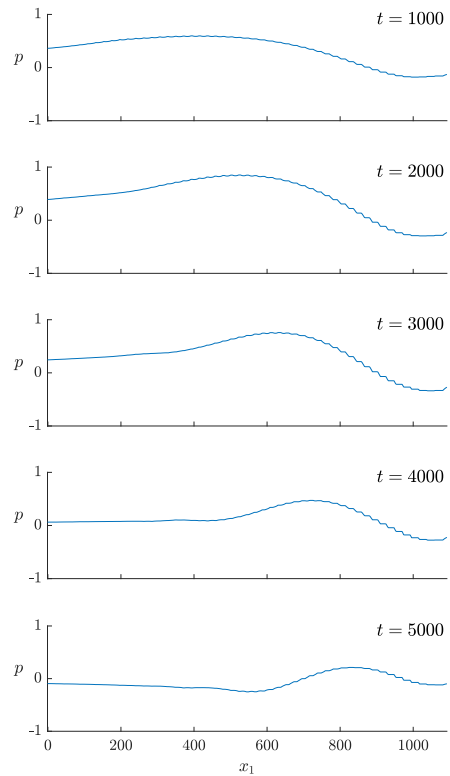}
	\caption{Our graded resonance model exhibits travelling wave behaviour in the pressure field. We show the evolution over time of the acoustic pressure $p=p((x_1,0),t)$ scattered by 50 evenly spaced circular resonators. 
	The acoustic pressure is initially zero then the resonators are simultaneously excited at $t=0$. We plot the cross-section of the field along $x_2=0$ (through the centres of the resonators).}
	\label{fig:travellingwave}
	\vspace{-10pt}
\end{wrapfigure}

It should also be noted that $|\alpha_n(\omega_n)|$ decreases in $n$. If we considered higher order resonances the corresponding constants would be significantly smaller. This justifies our choice to approximate $p$ as an element of $Y$ in \eqref{eq:decomposition_time} (\emph{i.e.} to only consider the $N$ subwavelength modes).

\subsection{Travelling waves} \label{sec:cochlea}

In trying to resolve the differences between the two main classes of cochlear model a crucial realisation is that our (resonance) model for the acoustic pressure exhibits the travelling wave behaviour. This is easy to see in models based on graded arrays of \textit{uncoupled} resonators, since a resonator's response time increases with decreasing characteristic frequency \cite{fletcher1992acoustic, bell2012resonance}, but is also true of our hybridised model.

Simultaneously exciting a graded array that is initially at rest produces the evolution shown in \Cref{fig:travellingwave}. The existence of a wave travelling from the small high-frequency resonators at the base of the cochlea to the larger low-frequency resonators at the apex is clear. This wave is the movement of the position of maximum acoustic pressure along the array of resonators. It is a consequence of the asymmetric eigenmodes (shown in Figures~\ref{fig:tonotopic}a-d) growing from rest at different rates. 

The parallels between the travelling wave in \Cref{fig:travellingwave} and that observed (\emph{e.g} by B\'ek\'esy) on the basilar membrane are clear. While it is true that acoustic waves enter the cochlea at the base and travel through the fluid to the apex, the wave observed by B\'ek\'esy moves much more slowly than this. The speed of sound in cochlear fluid is approximately 1500\si{\metre\per\second} whereas the travelling wave is observed at speeds close to 10\si{\meter\per\second}\cite{bell2012resonance, donaldson1993derived}. This justifies the choice to assume that all the resonators are excited simultaneously by an incoming signal \cite{bell2012resonance}. Since pressure changes in the fluid and the motion of the membrane will be physically linked, it is not surprising that the wave in \Cref{fig:travellingwave} shares a number of characteristics with B\'ek\'esy's observations. For instance, the amplitude initially grows before quickly diminishing and the wave is seen to slow as it moves through the array \cite{fletcher1992acoustic, von1960experiments, donaldson1993derived}. It should also be noted that travelling waves have been observed in the cochlear fluid (as predicted here) as well as on the membrane \cite{olson1999direct}.

\subsection{Tonotopic map} \label{sec:tonotopic}

B\'ek\'esy's famous experiments further revealed the existence of a relationship between signal frequency and the position in the cochlea where the sound is most strongly detected. His results showed that the frequency $f(x)$ giving rise to maximum excitation at a distance $x$ from the base of the cochlea satisfies a tonotopic map of the form
\begin{equation} \label{eq:tonotopic_map}
f(x) = ae^{-x/d}+c,
\end{equation}
for some $a,d,c\in\mathbb{R}$ \cite{von1960experiments}. In \Cref{fig:tonotopic} we show the relationship between the position of maximum amplitude of each eigenmode and the associated resonant frequency. We see that, if some of the lowest frequency modes are ignored, the pattern follows a relationship that is approximately of the form \eqref{eq:tonotopic_map} (with $a=0.0126, d = -0.0117, c = 0.0060$).
The eigenmodes shown in Figures~\ref{fig:tonotopic}b-d demonstrate the basis for the tonotopic map. Each features oscillations with a clear peak followed by a rapid decrease in amplitude (which explains the growth and then rapid decay of the travelling wave that was observed in \Cref{fig:travellingwave}).

It is not clear why the lowest frequency modes (\emph{e.g.} Figure~\ref{fig:tonotopic}a) do not fit the pattern that is established by the majority of the eigenmodes, or what the implications of this could be. However the relatively large negative imaginary parts of the associated resonant frequencies mean this phenomenon has a less significant impact on the evolution of the acoustic pressure field.


\begin{figure}
	\begin{tikzpicture}
	\node[anchor=south west,inner sep=0] (image) at (0,0) {\includegraphics[width=\textwidth]{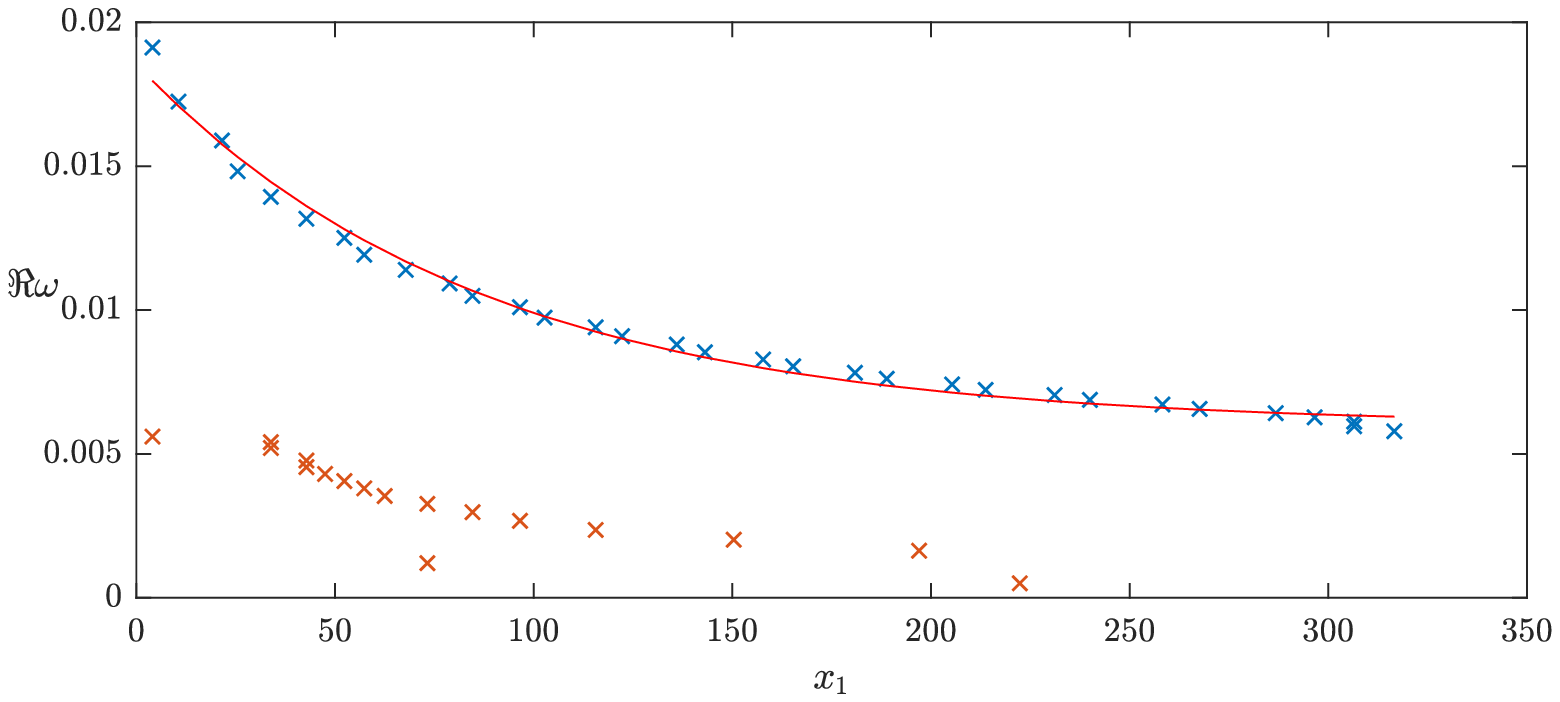}};
	\node[] at (3.7,1.9) {\footnotesize a};
	\node[] at (11.3,2.5) {\footnotesize b};
	\node[] at (7.5,2.8) {\footnotesize c};
	\node[] at (4.5,3.5) {\footnotesize d};
	\end{tikzpicture}\\
	\centering
	\subfigure[$\omega_{10}=0.003803-0.000003i$]{
		\includegraphics[width=0.45\textwidth]{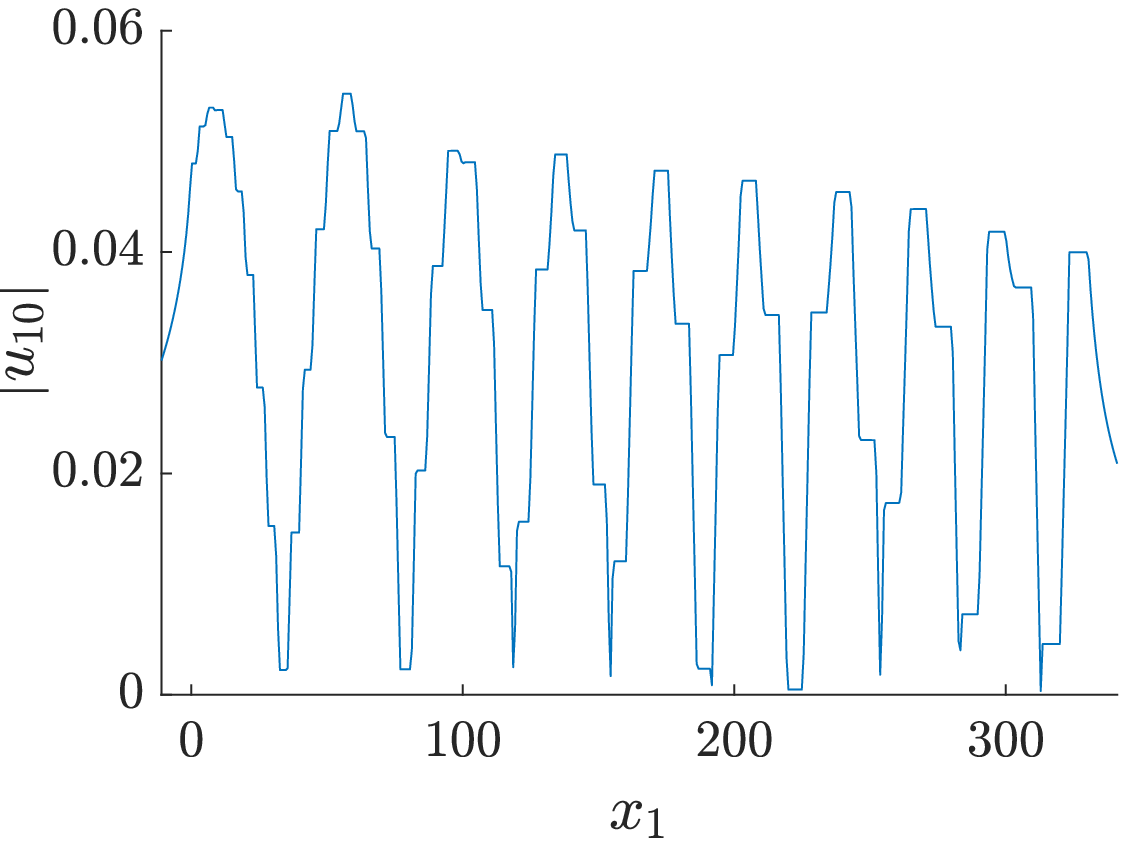}
	}
	\subfigure[$\omega_{20}=0.006122-0.000001i$]{
		\includegraphics[width=0.45\textwidth]{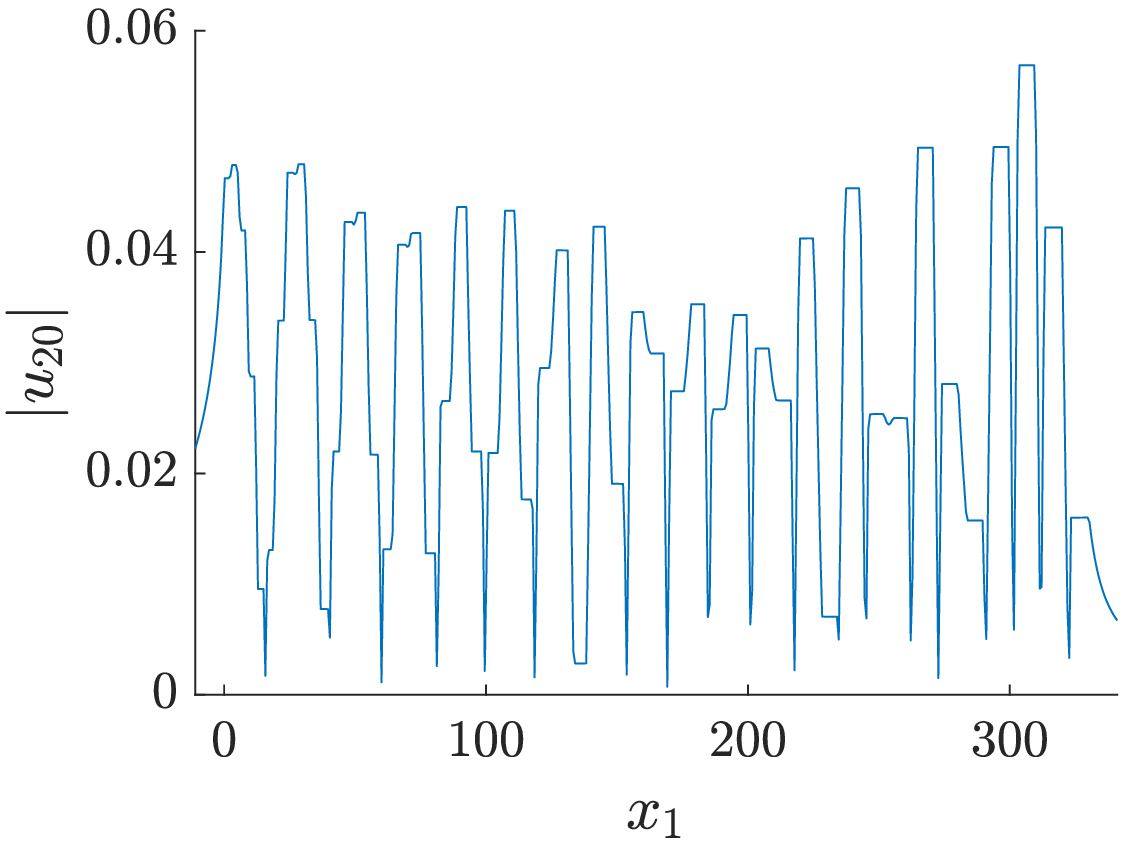}
	}\\
	\subfigure[$\omega_{30}=0.007612-0.000001i$]{
		\includegraphics[width=0.45\textwidth]{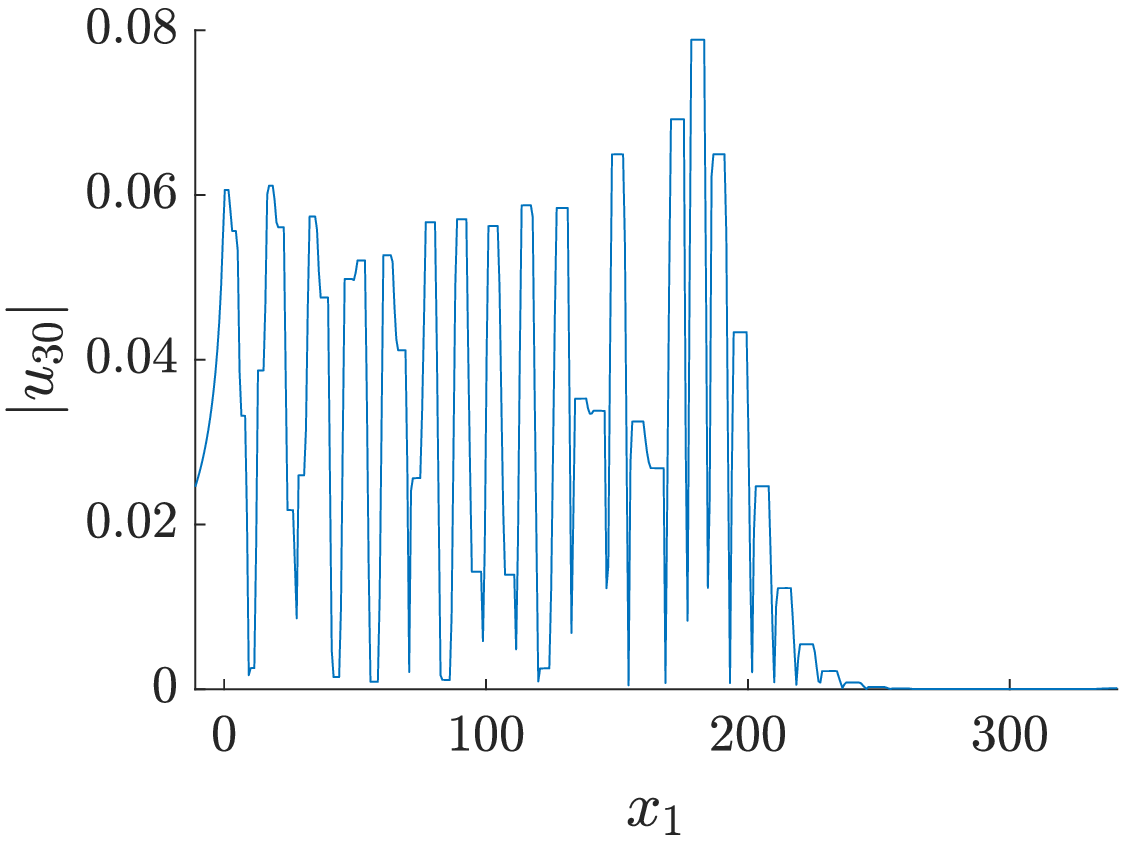}
	}
	\subfigure[$\omega_{40}=0.01049-0.000002i$]{
		\includegraphics[width=0.45\textwidth]{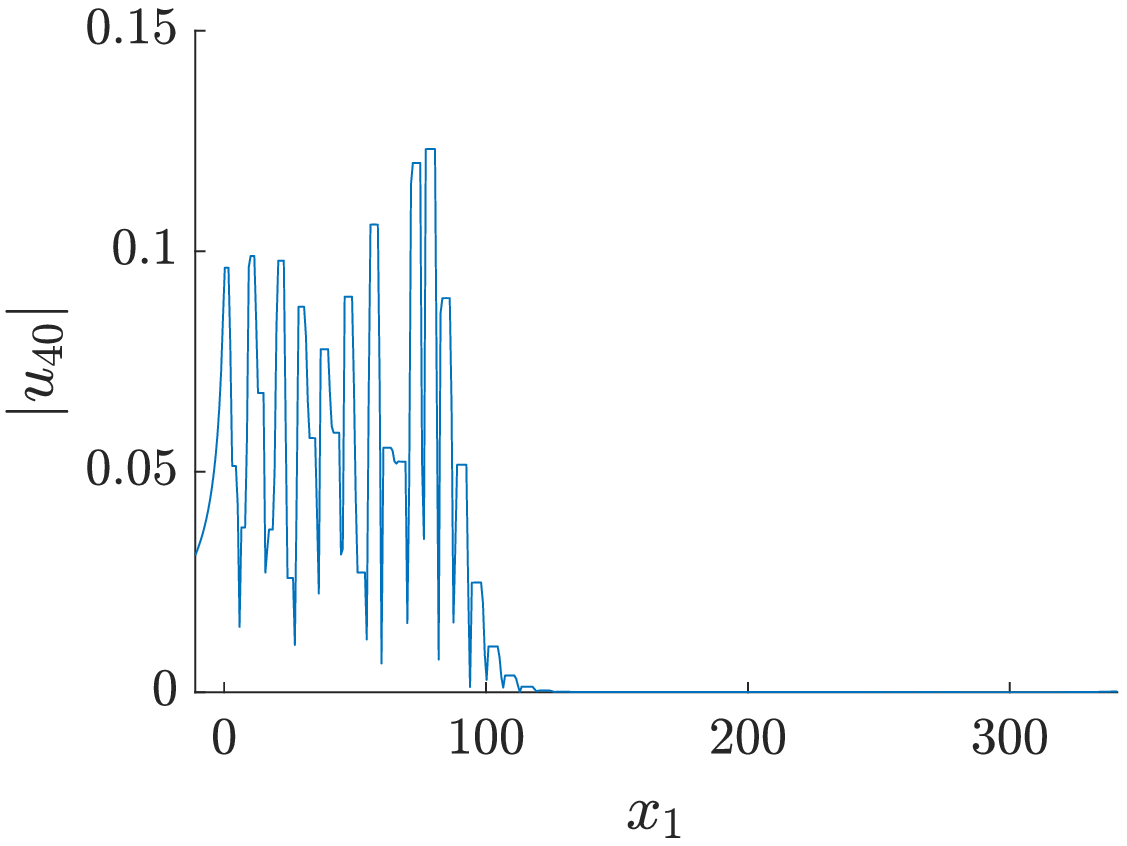}
	}
	\caption{ The existence of a tonotopic map for a passive system of graded oscillators. The top plot shows, for each eigenmode, the relationship between the real part of the associated resonant frequency $\Re\omega$ and the location ($x_1$-coordinate) of the maximum amplitude. We study the case of 50 resonators, increasing in size by a factor of 1.05 from left to right. A (least squares) approximation to the relationship exhibited by the blue points is shown, this has equation $0.0126e^{-0.0117x}+0.0060$. The 17 orange points are excluded from this calculation.
		\\
		(a)-(d) are the eigenmodes corresponding to the points marked on the top plot. We depict the absolute value of each eigenmode $|u_n|=|u_n(x_1,0)|$ along the line $x_2=0$ (through the centres of the resonators). It should be noted that the eigenmodes quickly decrease to zero outside of the region where the resonators are located.
		\\
	} \label{fig:tonotopic}
\end{figure}

\newpage
\section{Concluding remarks}

In this paper, we have modelled the passive cochlea by approximating its response by a graded array of subwavelength resonators. We have used boundary integral methods to compute leading order approximations to the resonant frequencies and associated eigenmodes this fully-coupled system. This model has the ability to decompose incoming signals into these resonant modes. As the number of resonators is increased, the resonant frequencies densely fill a finite range meaning that a large system can capture signals with a high frequency resolution. 

It is a significant observation that a simple graded-resonance model, with appropriate coupling between the resonators, predicts travelling wave behaviour in the acoustic pressure field, and that this has similar characteristics to that observed in the membrane motion. 
In some sense, these models represent the unification of Helmholtz' and B\'ek\'esy's ideas \cite{bell2012resonance, bell2018perspectives}.


It is well known that the cochlea is an active organ and even emits sounds (known as otoacoustic emissions) as part of its response to a signal \cite{kemp1978stimulated, kemp2008otoacoustic, crawford1985mechanical, moller2000hearing, shera2003mammalian}. For instance, a key feature that our current model lacks is the ability to amplify quiet sounds more greatly than louder ones. Such non-linear amplification is needed in order to account for the ear's remarkable ability to hear sounds over a large range of amplitudes. In this work, we have only considered a passive system of resonators but have presented a model which, by introducing appropriate non-linear forcing terms in \eqref{eq:wave_equation}, can be modified to include active elements in future work.


\vspace{0.7cm}

\noindent The code developed for this study is available online at

 \url{https://github.com/davies-b/cochlea_passive} 

\section*{Acknowledgements} 
The authors are grateful to Fabrice Lemoult for their participation in valuable discussions and to Andrew Bell for insightful comments made on an early version of this manuscript.

\bibliographystyle{apa}
\bibliography{cochlea}

\appendix
\section{Material parameters} \label{appen:high_contrast}

Here, we estimate the appropriate material parameters for the system of subwavelength resonators used in our version of the model from \cite{babbs2011quantitative}. In particular, using physical values for the cochlea and representations of subwavelength acoustic resonators as harmonic oscillators, we estimate the appropriate value for the bulk modulus contrast $\mu$, defined in \eqref{defn:contrasts}.

In \cite{babbs2011quantitative} it is shown that a suitable approximation to basilar membrane motion can be achieved by considering an array of masses on springs. For a piece of membrane with area $A$, thickness $h$, width $w$ and Young's modulus $E$, the spring constant of the equivalent resonator is shown to be given by
\begin{equation} \label{eq:membrane_formula}
K = C\frac{EAh^3}{w^4},
\end{equation}
where $C\approx30$ is a dimensionless constant.

In \cite{devaud2008minnaert} it is shown that a subwavelength acoustic resonator behaves as a one-degree-of-freedom harmonic oscillator. In the case where the resonator $D_i$ is spherical with radius $R$, it is shown that its stiffness is given by
\begin{equation} \label{eq:bubble_formula}
K = 12\pi \kappa_b R.
\end{equation}

Combining \eqref{eq:membrane_formula} and \eqref{eq:bubble_formula}, we see that the contrast $\mu$ is given by
\begin{equation}
\mu = \frac{C}{12\pi} \frac{EAh^3}{w^4} \frac{1}{\kappa}. 
\end{equation}
In \Cref{table:values}, we give values for the relevant material parameters, derived by experimentalists working on biological cochleas. Using the orders of magnitude of these values we find that, if the membrane is approximated by $N=O(10^2)$ subwavelength resonators, then it should hold that 
\begin{equation*}
\mu \approx O(10^{-8}).
\end{equation*}

\begin{table}[h]
	\footnotesize
	\begin{center}
		\begin{tabular}{m{0.4\textwidth} m{0.4\textwidth}  }
			\hline
			Quantity &  Approximate Value 
			\\ \hline
			$L$: length of uncoiled cochlea & 3.5cm \\
			$w$: width of basilar membrane & 0.015cm at base to 0.056cm at apex \\ 
			$h$: average thickness of basilar membrane  & 0.002cm \\
			$r$: average radius of scalae & 0.1cm \\ 
			$E$: Young's modulus of basilar membrane & $10^8$\si{\newton\per\meter\squared} at base to $10^7$\si{\newton\per\meter\squared} at apex \\
			$\kappa$: bulk modulus of water & $2\times10^9$\si{\pascal} \hfill \cite{mcallister2013pipeline}
		\end{tabular}
	\end{center}
	\caption{Approximate values for the material parameters of a biological cochlea. Unless specified, the values are taken from \cite{babbs2011quantitative}.} \label{table:values}
\end{table}

\end{document}